\documentclass[10pt]{amsart}
\usepackage{amsmath}
\usepackage{amssymb}
\usepackage{amsbsy}
\usepackage{amsfonts}

\newtheorem{thm}{Theorem}[section]
\newtheorem{lem}[thm]{Lemma}
\newtheorem{prop}[thm]{Propsition}
\newtheorem{cor}[thm]{Corollary}

\numberwithin{equation}{section}

\theoremstyle{remark}
\newtheorem{rem}{Remark}

\newcommand{\hlp}{H_{\theta, p}}
\newcommand{\hl}{H_{\theta}}

\newcommand{\xl}{X_{\theta}}

\begin{document}

\title[Scattering of 2D ICQNLS]{Small data scattering of the INHOMOGENEOUS CUBIC-QUINTIC NLS in 2 dimensions}

\author[Y. Cho]{Yonggeun ChO}
\address{Department of Mathematics, and Institute of Pure and Applied Mathematics, Chonbuk National University, Jeonju 561-756, Republic of Korea}

\author[K. Lee]{Kiyeon Lee}
\address{Department of Mathematics, Chonbuk National University, Jeonju 561-756, Republic of Korea}
\email{leeky@jbnu.ac.kr}

\begin{abstract}
The aim of this paper is to show the small data scattering for 2D ICQNLS:
  $$
  iu_t=-\Delta u + K_1(x)|u|^2u+K_2(x)|u|^4u.
  $$
Under the assumption that $\left| \partial^j K_l \right| \lesssim |x|^{b_l -j}$ for $j=0, 1, 2, l=1, 2$ and $0 \le b_l \le l - \frac23$, we prove the small data scattering in an angularly regular Sobolev space $H_\theta^{1,1}$. We use the decaying property of angularly regular functions, which are defined as functions in Sobolev space $H_\theta^{1, 1} \subset H^1$ with angular regularity such that $\|\partial_\theta f\|_{H^1} < \infty$, and also use the recently developed angularly averaged Strichartz estimates \cite{stri2, cholee, ghn}. In addition, we suggest a sufficient condition for non-existence of scattering.
\end{abstract}


\maketitle

\section{Introduction}

In this paper we consider the following Cauchy problem for inhomogeneous cubic-quintic nonlinear Schr\"odinger equations (ICQNLS):
\begin{eqnarray}\label{eq:sch}
  \left\{ \begin{array}{c}
           i u_{t} = -\Delta u + K_{1}(x)Q_1(u) + K_{2}(x)Q_2(u)\phantom{1} \mathrm{in}~ \mathbb{R}^{1+2}, \\
           u(x,0)=\varphi(x),
         \end{array} \right.
\end{eqnarray}
where $Q_1(u)=|u|^2 u,Q_{2}(u)=|u|^4 u$, and $K_1 , K_2 \in C^2(\mathbb{R}^2 \backslash \{0\};\mathbb{C})$. The model of ICQNLS \eqref{eq:sch} can be a dilute BEC when both the two- and three-body interactions of the condensate are considered. For this see \cite{bpvt, ts} and the references therein. Also it has been considered to study the laser guiding in an axially nonuniform plasma channel. For this see \cite{gill, ss}.

The interaction coefficients $K_l$ are assumed to satisfy the growth condition: for some constants $b_1,b_2 \ge 0$
\begin{eqnarray}
  \left| \partial^j K_l \right| \lesssim |x|^{b_l -j}, j=0, 1, 2,\; l=1, 2,\label{kl-up}
\end{eqnarray}
where $\partial$ is one of $\partial_j,j=1,2$. Some basic notations are listed at the end of this section.

By Duhamel's formula, the equation \eqref{eq:sch} is written as an integral equation
\begin{eqnarray}
  u = e^{it\Delta}\varphi-i\int_{0}^{t}e^{i(t-t')\Delta}\left[ K_1(x)Q_1(u(t'))+K_2(x)Q_2(u(t')) \right]dt'.
\end{eqnarray}
Here we define the linear propagator $e^{it\Delta}$given by the solution to the linear problem $i\partial_tv=-\Delta v$ with initial data $v(0)=f$. It is formally given by
$$e^{it\Delta}f = \mathcal{F}^{-1}\left(e^{-it|\xi|^2}\mathcal{F}(f)\right)= (2\pi)^{-2}\int_{\mathbb{R}^2} e^{i(x\cdot\xi - t|\xi|^2)}\widehat{f}(\xi)d\xi,$$
where $\widehat{f} = \mathcal F( f)$ denotes the Fourier transform of $f$ and $\mathcal F^{-1}(g)$ the inverse Fourier transform of $g$ such that
$$\mathcal{F}(f)(\xi) = \int_{\mathbb{R}^2} e^{- ix\cdot \xi} f(x)\,dx,\quad \mathcal F^{-1} (g)(x) = (2\pi)^{-2}\int_{\mathbb{R}^2} e^{ix\cdot \xi} g(\xi)\,d\xi.$$

If $K_l$ are real-valued, then we can define the mass and energy of the solution $u$ to \eqref{eq:sch} as follow:
\begin{align*}
m(u(t))&:= ||u(t)||_{L_x^2}^2,\\
E(u(t))&:= \frac12||\nabla u(t)||_{L_x^2}^2 + \frac14\int K_1(x)|u(t,x)|^4 dx + \frac16\int K_2(x)|u(t,x)|^6 dx.
\end{align*}
We say that mass and energy of solution $u$ are conserved if they are constant w. r. t. time.

The inhomogeneous NLS of single nonlinearity with coefficient behaving like $|x|^b (b \in \mathbb R)$ have been extensively studied by the authors of \cite{cg, che, choleem, coge, fa, fagu, fiwa, fuoh, ge, gest, guz, mer, rs, zhu}. In particular, the well-posedness for the coefficient with $b > 0$ has been considered under radial symmetry (\cite{cg,che,zhu}). The radial symmetry plays a crucial role in nonlinear estimates in the energy space thanks to the decaying properties of radial Sobolev functions. Recently the first author of this paper considered 3D ICQNLS without radial symmetry in \cite{3d}, where the well-posedness, finite time blowup and small data scattering are systematically studied based on the decaying properties of angularly regular functions and 3D endpoint Strichartz estimate. Encouraged by this work, we consider 2D scattering problem in this paper. One of the most distinguished feature of 2D scattering problem is that the endpoint Strichartz estimate is forbidden. What is worse is that we cannot control the cubic term to the long time with the usual Strichartz estimates of admissible Schr\"odinger pairs since the coefficient $K_1$ makes the equation \eqref{eq:sch} mass-subcritical in terms of $L^2$-scaling. That is to say a time factor appears and it prevents us from handling solutions globally in time due to the lack of uniform bound of angular Sobolev norm. For this see the local well-posedness in Section $3$. In this paper we will overcome these obstacles. To this end we will use not only decaying property of angularly regular functions but also recently developed angularly averaged Stricharz estimates \cite{stri2, cholee, ghn}, which provides us an extended range of admissible pairs. We will see the detail in Section 4 below.

The angularly regular functions are defined by the angular derivative $\partial_\theta$, where $\theta$ is the argument such that $x = (x_1,x_2) = (|x|\cos\theta, |x|\sin\theta)$. Since $\partial_\theta = x \times \nabla = x_1\partial_2 - x_2\partial_1$, the operator $-i\partial_\theta$ is also referred as angular momentum. Now we define Sobolev spaces $\hlp^{1, 1}$, $\hlp^{2, 1} (1 \le p \le \infty)$ associated with $\partial_\theta$ as follows:
\begin{align*}
H_{\theta, p}^{1, 1} = \{f \in  H_p^1 : \|f\|_{H_{\theta, p}^{1, 1}} = \|(f, \partial_\theta f)\|_{H_p^1} < \infty\}, \quad \hlp^{2, 1} &= \hlp^{1, 1} \cap H_p^2.
\end{align*}
Here $\|(g, \partial_\theta g)\|_{Y}$ denotes $\|g\|_{Y} + \|\partial_\theta g\|_{Y}$ for Banach spaces $Y$ on $\mathbb R^2$, and $H_p^{n}$ denotes the standard $L_x^p$ Sobolev space. If $p = 2$, then we drop the exponent $p$ and denote $ H_{\theta,\, 2}^{n, 1}$ by $\hl^{n, 1}$. Clearly the radial Sobolev space $H_{rad, p}^n $ is embedded in $H_{\theta, p}^{n, 1}$ and $\|f\|_{H_p^1} = \|f\|_{H_{\theta, p}^{n, 1}}$ for any radial function $f$. These spaces give us Sobolev type inequalities associated with angular regularity such as
\begin{align}\label{ang-sobo}
\||x|^bf\|_{L_x^\infty} \lesssim \|f\|_{\hl^{1, 1}}\;\;\mbox{for}\;\; 0 < b \le \frac{1}{2}.
\end{align}
This estimate is crucial for our nonlinear estimates. See Lemma \ref{lem;ang-soin} below.

We first introduce the local result. The equation (\ref{eq:sch}) is said to be locally well-posed (LWP) in $H_{\theta}^{n}$ if there exist maximal interval $I_* = (-T_*,T^*)$ and a unique solution $u \in C(I_*,H_{\theta}^{n})$ with continuous dependency on the initial data and blowup alternative $\left(T^* < \infty \Rightarrow \lim_{t \to T^*}||u(t)||_{H_{\theta}^{n}}=\infty  \right)$. By the contraction argument based on the usual Strichartz estimates and the inequality \eqref{ang-sobo} we have the following local result.
%%%%%%%%%%%%%%%%%%%%%%%%%%%%%%%%%%%%%%%%%%%%%%%%%%%%%%%%%%%%%%%%%%%%%%%%%%%%%%%%%%%%%%%%%%%%%%%%%%%%%%%%%%%%%%%%%%%%%%%%%%%%%%%%%%%%%%%%%%%%%%%%%%%%%%%%%%%%%%%%%%%%%%%%%%%%%%%%%
%%%%%%%%%%%%%%%%%%%%%%%%%%%%%%%%%%%%%%%%%%%%%%%%%%%%%%%%%%%%%%%%%%%%%%%%%%%%%%%%%%%%%%%%%%%%%%%%%%%%%%%%%%%%%%%%%%%%%%%%%%%%%%%%%%%%%%%%%%%%%%%%%%%%%%%%%%%%%%%%%%%%%%%%%%%%%%%%%
\begin{prop}\label{thm;lwp}
{\bf(1) Local well-posedness} If $0 \leq b_1 \leq \frac{1}{2}$ and $0 \leq b_2 \leq \frac{3}{2}$, then (\ref{eq:sch}) is locally well-posedness in $H_{\theta}^{n, 1}$, $n = 1, 2$.
Moreover, if $K_1,K_2$ are real-valued, then mass and energy are conserved.\\
{\bf (2) Blowup criterion} Let $K_1$ and $K_2$ be real-valued functions as above, and satisfy the rigidity condition that
\begin{align}\label{rigid}
-x\cdot \nabla K_1 \le \alpha K_1\;\; \mbox{and}\;\; 2K_2 - x\cdot \nabla K_2 \le \alpha K_2
\end{align}
 for some $\alpha  \ge 0$. Suppose that $E(\varphi) < 0$ and $x\varphi \in L_x^2$. Then the solutions blow up in finite time.\\
{\bf (3) Radial global well-posedness} Let $K_l$, $l = 1, 2$ be radially symmetric functions. Suppose that $K_l \ge 0$ and $\varphi \in H_{rad}^1$ or that $K_l$ may be negative or change their sign, and $\|\varphi\|_{H_{rad}^1}$ is sufficiently small. Then \eqref{eq:sch} is globally well-posed in $H_{rad}^1$.
\end{prop}
%%%%%%%%%%%%%%%%%%%%%%%%%%%%%%%%%%%%%%%%%%%%%%%%%%%%%%%%%%%%%%%%%%%%%%%%%%%%%%%%%%%%%%%%%%%%%%%%%%%%%%%%%%%%%%%%%%%%%%%%%%%%%%%%%%%%%%%%%%%%%%%%%%%%%%%%%%%%%%%%%%%%%%%%%%%%%%%%%
Since our method relies heavily on the angular Sobolev estimate \eqref{ang-sobo}, we cannot control  without further assumption the growing coefficients of nonlinear term for present in case that $b_1 > \frac12$ or $b_2 > \frac32$. If $b_1 = b_2 = 0$, then \eqref{eq:sch} can be readily shown to be locally well-posed in $H^1$. See Remark \ref{lwp-in-h1} below.
%%%%%%%%%%%%%%%%%%%%%%%%%%%%%%%%%%%%%%%%%%%%%%%%%%%%%%%%%%%%%%%%%%%%%%%%%%%%%%%%%%%%%%%%%%%%%%%%%%%%%%%%%%%%%%%%%%%%%%%%%%%%%%%%%%%%%%%%%%%%%%%%%%%%%%%%%%%%%%%%%%%%%%%%%%%%%%%%%
%%%%%%%%%%%%%%%%%%%%%%%%%%%%%%%%%%%%%%%%%%%%%%%%%%%%%%%%%%%%%%%%%%%%%%%%%%%%%%%%%%%%%%%%%%%%%%%%%%%%%%%%%%%%%%%%%%%%%%%%%%%%%%%%%%%%%%%%%%%%%%%%%%%%%%%%%%%%%%%%%%%%%%%%%%%%%%%%%
For the proof of blowup criterion we prove the localized virial argument for which the weighted condition $|x|\varphi\in L_x^2$ and the rigidity condition \eqref{rigid} of $K_l$ are necessary. The rigidity condition of coefficients $K_l$ is not sharp and can be modified. Once a regular solution exists even for the case that $b_1 > \frac12$ or $b_2 > \frac32$, the finite time blowup can be shown to occur by the same virial argument.
%%%%%%%%%%%%%%%%%%%%%%%%%%%%%%%%%%%%%%%%%%%%%%%%%%%%%%%%%%%%%%%%%%%%%%%%%%%%%%%%%%%%%%%%%%%%%%%%%%%%%%%%%%%%%%%%%%%%%%%%%%%%%%%%%%%%%%%%%%%%%%%%%%%%%%%%%%%%%%%%%%%%%%%%%%%%%%%%%
%%%%%%%%%%%%%%%%%%%%%%%%%%%%%%%%%%%%%%%%%%%%%%%%%%%%%%%%%%%%%%%%%%%%%%%%%%%%%%%%%%%%%%%%%%%%%%%%%%%%%%%%%%%%%%%%%%%%%%%%%%%%%%%%%%%%%%%%%%%%%%%%%%%%%%%%%%%%%%%%%%%%%%%%%%%%%%%%%
Under the radial symmetry one can perform the standard time iteration method through the uniform bound of $H^1$ norm, which is guaranteed by energy conservation and time continuity argument. We refer the readers to Section 3.3.
%%%%%%%%%%%%%%%%%%%%%%%%%%%%%%%%%%%%%%%%%%%%%%%%%%%%%%%%%%%%%%%%%%%%%%%%%%%%%%%%%%%%%%%%%%%%%%%%%%%%%%%%%%%%%%%%%%%%%%%%%%%%%%%%%%%%%%%%%%%%%%%%%%%%%%%%%%%%%%%%%%%%%%%%%%%%%%%%%

%%%%%%%%%%%%%%%%%%%%%%%%%%%%%%%%%%%%%%%%%%%%%%%%%%%%%%%%%%%%%%%%%%%%%%%%%%%%%%%%%%%%%%%%%%%%%%%%%%%%%%%%%%%%%%%%%%%%%%%%%%%%%%%%%%%%%%%%%%%%%%%%%%%%%%%%%%%%%%%%%%%%%%%%%%%%%%%%%
Now we state our main result, small data scattering.
%\begin{defn}\label{defn;scattering}
We say that a solution $u$ to \eqref{eq:sch} scatters (to $u_\pm$) in a Hilbert space $\mathcal X$ if there exist  $\varphi_\pm \in \mathcal X$ $($with $u_\pm(t) = e^{it{\Delta}}\varphi_\pm)$ such that $\lim_{t  \to \pm\infty}\|u(t) - u_\pm\|_{\mathcal X} = 0$	
%\end{defn}
We have the following.
\begin{thm}\label{thm;scattering}
	Let $0 \le b_1 < \frac13$ and $0 \le b_2 < \frac43$.  If $\|\varphi\|_{H_{\theta}^{1, 1}}$ is sufficiently small, then there exists a unique $u \in (C \cap L^\infty)(\mathbb R ; H_{\theta}^{1, 1})$ to \eqref{eq:sch} which scatters in $\hl^{1,1}$.
\end{thm}
The main ingredient of the proof is how to control globally in time the mass-subcritical nature of cubic term occurring due to the growing coefficient. This can be settled by using the angularly averaged Strichartz estimates \eqref{str-besov1} and \eqref{str-besov3}, which hold for a wider range of admissible pairs. Let us brief on the key steps. Taking a angular and space derivatives to the Duhamel part and then applying the extended Strichartz estimates, we find several nonlinear estimates. One of the most significant part of nonlinear estimate is $\|(\partial_\theta K_1)|u|^2|\partial u|\|_{L_t^1L_x^2(|x|\ge 1)} (0 < b_1 < \frac13)$, which can be bounded by $$\||x|^{b_1}u\|_{L_t^\infty L_\rho^\frac{2r}{r-4}L_\theta^\infty}\|u\|_{L_t^2L_\rho^rL_\theta^2}\|\partial u\|_{L_t^2L_\rho^rL_\theta^\infty}$$ for some $r > 6$ and near $6$. We control the first norm by Sobolev inequality in $H_\theta^{1,1}$ of Corollary \ref{decay} below and hence we need the condition $b_1 < \frac13$. We control the second and third norms by angularly averaged Strichartz estimates for the pair $(2, r)$. As for the quintic term we find $\|(\partial_\theta K_2)|u|^4|\partial u|\|_{L_t^1L_x^2(|x| \ge 1)}$ ($b_2 > 1$), which can be easily bounded by $\||x|^\frac12u \|_{L_t^\infty L_x^\infty}^2\||x|^{b_2-1}|u|^2|\partial u|\|_{L_t^1L_x^2(|x|\ge 1)}$. We use \eqref{ang-sobo} for the first norm and the previous cubic estimate for the later, for which the condition $b_2 < \frac43$ is necessary.
%%%%%%%%%%%%%%%%%%%%%%%%%%%%%%%%%%%%%%%%%%%%%%%%%%%%%%%%%%%%%%%%%%%%%%%%%%%%%%%%%%%%%%%%%%%%%%%%%%%%%%%%%%%%%%%%%%%%%%%%%%%%%%%%%%%%%%%%%%%%%%%%%%%%%%%%%%%%%%%%%%%%%%%%%%%%%%%%%

%%%%%%%%%%%%%%%%%%%%%%%%%%%%%%%%%%%%%%%%%%%%%%%%%%%%%%%%%%%%%%%%%%%%%%%%%%%%%%%%%%%%%%%%%%%%%%%%%%%%%%%%%%%%%%%%%%%%%%%%%%%%%%%%%%%%%%%%%%%%%%%%%%%%%%%%%%%%%%%%%%%%%%%%%%%%%%%%%
On the other hand, in view of the previous work of \cite{barab, choz} one can expect a non-existence of scattering if $b_1$ is large. The reason is that a large $b_1$ leads us to an effect of the long-range scattering. Here we give a sufficient conditions as follows.
\begin{thm}\label{thm;nonscattering}
	Assume that $K_1(x)=|x|^{b_1}$ and $K_2(x)=|x|^{b_2}$ for $b_1 \ge 1,0 \le b_2 \le 2+b_1$. Let $u$ be a global smooth solution of \eqref{eq:sch} with $xu \in C(\mathbb R; L_x^2)$, which scatters to $u_{\pm}=e^{it\Delta}\varphi_\pm$ in $L_x^2$ for some smooth function $\varphi_\pm$. Then $u,u_\pm \equiv 0$.
\end{thm}
We prove by contradiction. For this purpose we develop a pseudo-conformal identity through the virial identity of Lemma \ref{loc-virial} and show time decay of potential energy such that
$$
\frac14\int K_1|u|^4\,dx + \frac16\int K_2|u|^6\,dx \lesssim t^{-(2 - b_1)}\;\;\mbox{if}\;\;t \gg 1.
$$
This estimate together with $L^2$ scattering enables us to handle the time decay of quintic term and hence to obtain an estimate
$$
-{\rm Im}\int u(t)\overline{u_+(t)}\,dx \gtrsim \int_1^t\tau^{-(2-b_1)}\,d\tau,
$$
provided $u, u_{\pm}$ are not identically $0$ and $u_\pm$ is sufficiently smooth. The RHS goes to infinity as $t \to \infty$, which contradicts to the uniform boundedness of the LHS.
%%%%%%%%%%%%%%%%%%%%%%%%%%%%%%%%%%%%%%%%%%%%%%%%%%%%%%%%%%%%%%%%%%%%%%%%%%%%%%%%%%%%%%%%%%%%%%%%%%%%%%%%%%%%%%%%%%%%%%%%%%%%%%%%%%%%%%%%%%%%%%%%%%%%%%%%%%%%%%%%%%%%%%%%%%%%%%%%%

%%%%%%%%%%%%%%%%%%%%%%%%%%%%%%%%%%%%%%%%%%%%%%%%%%%%%%%%%%%%%%%%%%%%%%%%%%%%%%%%%%%%%%%%%%%%%%%%%%%%%%%%%%%%%%%%%%%%%%%%%%%%%%%%%%%%%%%%%%%%%%%%%%%%%%%%%%%%%%%%%%%%%%%%%%%%%%%%%
The authors would like to remark that it would be of interest to find out whether the small data scattering holds in $\frac13 \le b_1 < 1$ or not.
%%%%%%%%%%%%%%%%%%%%%%%%%%%%%%%%%%%%%%%%%%%%%%%%%%%%%%%%%%%%%%%%%%%%%%%%%%%%%%%%%%%%%%%%%%%%%%%%%%%%%%%%%%%%%%%%%%%%%%%%%%%%%%%%%%%%%%%%%%%%%%%%%%%%%%%%%%%%%%%%%%%%%%%%%%%%%%%%%

%%%%%%%%%%%%%%%%%%%%%%%%%%%%%%%%%%%%%%%%%%%%%%%%%%%%%%%%%%%%%%%%%%%%%%%%%%%%%%%%%%%%%%%%%%%%%%%%%%%%%%%%%%%%%%%%%%%%%%%%%%%%%%%%%%%%%%%%%%%%%%%%%%%%%%%%%%%%%%%%%%%%%%%%%%%%%%%%%
Our paper is organized as follows. In Section 2 we introduce several basic lemmata on the Sobolev inequalities and Strichartz estimates. Section 3 is devoted to establishing the local theory. The small data scattering is treated in Section 4. In the last section, we discuss the non-scattering results.\\

\noindent\textbf{Notations.}

\noindent$\bullet$ Fractional derivatives: $D^s = (-\Delta)^\frac{s}2 = \mathcal F^{-1}|\xi|^s\mathcal F$, $\Lambda^s = (1- \Delta)^\frac{s}2 = \mathcal F^{-1}(1+|\xi|^2)^\frac s2\mathcal F$ for $s > 0$.\\

\noindent$\bullet$ Sobolev spaces: $ \dot{H}_r^s=D^{-s}L^r,~ \dot{H}^s = \dot{H}_2^s,~ H_r^s = \Lambda^{-s}L^r,~ H^s = H_2^s,~ L^r = L_x^r(\mathbb{R}^2) $ for $s\in \mathbb{R}$ and $1 \leq r\leq \infty$.\\

\noindent$\bullet$ Sobolev spaces on the unit circle: $H_{\theta, p}^1 = \{f \in L_\theta^2(0, 2\pi) : \|f\|_{H_{\theta, p}^1} := \|f\|_{L_\theta^p(0, 2\pi)} + \|\partial_\theta f\|_{L_\theta^p(0, 2\pi)} < \infty\}$, $1 \le p \le \infty$.\\

\noindent$\bullet$ Mixed-normed spaces: For a Banach space $X$, $u \in L_{I}^{q}X$ iff $u(t)\in X $ for a.e. $t \in I$ and $ \|u\|_{L_{I}^{q}X} := \|\| u(t)\|_X\|_{L_{I}^{q}}$. Especially, we denote $L_{I}^{q}L_{x}^{r} = L_{t}^{q} \left( I;L_{x}^{r}(\mathbb{R}^2)  \right),L_{I,x}^{q}=L_{I}^{q}L_{x}^{q}$ and $L_{t}^{q}L_{x}^{r} = L_{\mathbb{R}}^{q}L_{x}^{r}$. We define the mixed-normed space in polar coordinate by
$$
L_\rho^{p_1}H_{\theta, p_2}^1 := \left\{f : \left( \int_{(0,\infty)}\|f\|_{H_{\theta, p_2}^1}^{p_1}\rho\,d\rho \right)^\frac1{p_1} < \infty \right\}.\\
$$

\noindent$\bullet$ As usual different positive constants are denoted by the same letter $C$, if not specified. $A \lesssim B$ means that $A \leq CB$ for some $C > 0$. $A \sim B$ means that $A \lesssim B$ and $ B \lesssim A$.\\

\section{Preliminary lemmas}

We say that a pair $(q,r)$ is {\it admissible} if it satisfies that $2 \leq q, r \leq \infty$, $(q,r) \neq (2,\infty)$, and $\frac{1}{q} + \frac{1}{r}=\frac{1}{2}$.

\begin{lem}[\cite{stri}]\label{lem;stri} Let $(q,r)$ and $(\widetilde{q},\widetilde{r})$ be any \textit{admissible pair}. Then we have
\begin{align*}
\|e^{it\Delta} \varphi\|_{L_{t}^{q}L_{x}^{r}}  &\lesssim  \|\varphi\|_{L_{x}^{2}}, \\
\|\int_{0}^{t} e^{i(t-t')\Delta}F dt' \|_{L_{t}^{q}L_{x}^{r}}  &\lesssim  \|F\|_{L_{t}^{\tilde{q}'}L_{x}^{\tilde{r}'}}.
\end{align*}
%{\bf (2)} If $f$ and $F$ are radially symmetric, then the endpoint is admissible.
\end{lem}

Taking $L_\theta^2$-average, we can extend the range of the Strichartz estimates as follows.
\begin{lem}[See \cite{stri2}, \cite{ghn}, and also \cite{cholee}]\label{lem;stri-extended} Let $6 < r \le \infty$ and let ${\rm supp} \;\widehat \varphi \subset \{\lambda/2 \le |\xi| \le 2\lambda\}$ for any $\lambda > 0$. Then we have
\begin{align*}
\|e^{it\Delta} \varphi\|_{L_{t}^{2}L_{\rho}^{r}L_\theta^2}  \lesssim \lambda^{-\frac 2r} \|\varphi\|_{L_{x}^{2}}.
\end{align*}
\end{lem}

The following is the well-known 2D Hardy-Sobolev inequality.
\begin{lem}\label{lem;soin} For any $f \in \dot{H}_p^s (0 < s < \frac2p, 2 < p < \infty)$ we have
$$
\||x|^{-s}f\|_{L_x^p} \lesssim \|f\|_{\dot{H}_p^s}.
$$
\end{lem}
This can be done by interpolation between Theorem 2 of \cite{ms} and critical Sobolev inequality $\|f\|_{BMO} \lesssim \|f\|_{\dot H^\frac 2p}$. For the later see  \cite{tr}.

Assuming further angular regularity, we have the space decay Sobolev inequality.
\begin{lem}\label{lem;ang-soin} Let $0 < b \le \frac{1}{2}$. Then for any $f\in H_{\theta}^{1,1}$ there holds
\begin{align*}
\|\left|x \right|^{b}f\|_{L_x^{\infty}} &\lesssim \|f\|_{H_{\theta}^{1,1}}.
\end{align*}
\end{lem}
If $b < \frac12$, then inequality is shown in \cite{fawa}. The case $b = \frac12$ can be shown by introducing a Besov space and its embedding.
In fact, we know from \cite{soin2} that
$$
\| |x |^{\frac{1}{2}}f\|_{L_x^{\infty}} \lesssim \|f\|_{\dot{B}_{\theta}^{\frac{1}{2}}}.
$$
Here $\dot B_\theta^{\frac12}$ is the Besov space with angular derivative defined by
$$
\dot{B}_{\theta}^{\frac12} := \{f \in  \dot{B}_{2,\,1}^{\frac12, 1} : \|f\|_{\dot{B}_{\theta}^\frac12} := \|f\|_{\dot{B}_{2,\, 1}^\frac12}  +  \|\partial_\theta f\|_{\dot{B}_{2,\, 1}^\frac12} < \infty\}.
$$
Hence Lemma \ref{lem;ang-soin} follow due to the embedding $\hl^{1,1} \hookrightarrow {\dot{B}}_{\theta}^{\frac{1}{2}}$.

\begin{cor}\label{decay}
Let $2 < p < \infty$. Then for any $f\in H_{\theta}^{1,1}$ there holds
\begin{align*}
\|\left|x \right|^{\frac12-\frac1p}f\|_{L_\rho^pL_\theta^{\infty}} &\lesssim \|f\|_{H_{\theta}^{1,1}}.
\end{align*}
\end{cor}
This can be readily shown by interpolating the estimates between Lemma \ref{lem;ang-soin} and the trivial estimate $\|f\|_{L_\rho^2L_\theta^\infty} \lesssim \|f\|_{H_\theta^{1,1}}$.

The final lemma is on the relation of angular derivative and radial operators.

\begin{lem}\label{lem;cm}
{\bf (1)} Let $s \geq 0$. Then $\partial_\theta D^sf=D^s\partial_\theta f$ and $\partial_\theta\Lambda^sf=\Lambda^s\partial_\theta f$	\\
{\bf (2)} Let $\psi$ be smooth and radially symmetric. Then $\partial_\theta(\psi*f) = \psi*(\partial_\theta f)$.
\end{lem}

\section{Well-posedness and blowup criterion}\label{sec;lwp}

Let $I_T =[-T,T]$. Let us define a complete metric spaces $X_\theta^{n, 1}(T,\delta),\;\; n = 1, 2$ with metric $d^{n, 1}$ by
$$
X_\theta^{n,1} (T,\delta ) := \{ u \in ( C\cap L^{\infty})(I_{T}; H_{\theta}^{n,1} ) : \|u \|_{L_{I_{T}}^{\infty}H_{\theta}^{n,1}} \leq \delta \}, \;d^{n,1}(u,v):= \|u - v\|_{L_{I_T}^{\infty}H_\theta^{n,1}}.
$$
Let $\mathcal{H}(u)$ be defined by $\mathcal{H}(u)=e^{it\Delta}\varphi + N(u)$ where $N(u)= -i \int e^{i(t-t')\Delta}\left[K_1Q_1(u)+K_2Q_2(u)\right]dt'$. We will show that  $\mathcal{H}$ is a contraction map on $X_\theta^{n,1}(T,\delta)$. By $N_{l}^{1,1}$ and $N_l^2$ for $l = 1, 2$ we denote the derivatives of Duhamel's part as follow:
\begin{align*}
N_{l}^{j, k}&= -i \partial^j {\partial_\theta}^k \int_{0}^{t}e^{i(t-t')\Delta}\left[K_lQ_l(u)\right]dt' ( j, k = 0, 1),\\
N_{l}^2&= -i \partial^2 \int_{0}^{t}e^{i(t-t')\Delta}\left[K_lQ_l(u)\right]dt'.
\end{align*}
From Leibniz rule and Lemma \ref{lem;cm} it follows that
\begin{align*}
N_{l}^{1, 1} &= -i \int_{0}^{t}e^{i(t-t')\Delta}\left[(\partial {\partial_\theta} K_l)Q_l(u)+ \partial_\theta K_l\partial Q_l(u)+(\partial K_l) \partial_\theta Q_l(u) + K_l\partial {\partial_\theta} Q_l(u)\right]dt',\\
N_l^2 &= -i \int_{0}^{t}e^{i(t-t')\Delta}\left[(\partial^2K_l)Q_l(u) + 2(\partial K_l)(\partial Q_l(u))+ K_l\partial^2Q_l(u)\right]dt'.
\end{align*}

We only consider three cases:  $ 0 < b_1 < \frac{1}{2}, 0 < b_2 < \frac{3}{2}$; $b_1 = 0$ and $b_2 = 0$;  $b_1=\frac12$ and $b_2=\frac32$. From these one can readily treat the remaining cases.

\subsection{Contraction on $X_\theta^{1,1}(T, \delta)$}
\subsubsection{Case: $ 0 < b_1 < \frac{1}{2}, 0 < b_2 < \frac{3}{2}$}\label{subs;1}
Given $\delta$, from Lemmas \ref{lem;stri}, \ref{lem;ang-soin}, and \ref{lem;ang-soin} we can obtain that for any $u\in X_\theta^{1,1}\left (T,\delta \right)$ and for $j, k = 0, 1$
$\|N_1^{j, k}\|_{L_{I_T}^{\infty}L_x^2} \lesssim (T+T^{\frac{4}{3}})\|u\|_{H_\theta^{1,1}}^3$ and $\|N_2^{j, k}\|_{L_{I_T}^{\infty}L_x^2} \lesssim (T+T^{\frac{4}{3}})\|u\|_{H_\theta^{1,1}}^5$. In fact, using the bound \eqref{kl-up}, we deduce that
$$
|K_l(x)| + |\partial_\theta K_l| \lesssim |x|^{b_l},\quad |\partial K_l(x)| + |\partial\partial_\theta K_l(x)| \lesssim |x|^{b_l-1}.
$$
We consider only two significant parts $|\partial K_l||u|^{2l} |\partial_\theta u|$ and $|K_l||u|^{l+1}|\partial u||\partial_\theta u|$. Choose the admissible pair $(\widetilde q, \widetilde r) = (4, 4)$. Then we have that for $0 < \varepsilon < \min(\frac12-b_1, \frac32-b_2)$
\begin{align*}
&\||x|^{b_1-1}|u|^2 |{\partial_\theta}u| \|_{L_{I_T}^{\frac{4}{3}}L_x^{\frac{4}{3}}} + \||x|^{b_1}|u||\partial u| |{\partial_\theta}u| \|_{L_{I_T}^{\frac{4}{3}}L_x^{\frac{4}{3}}}\\
&\qquad \lesssim T^{\frac{3}{4}}(\||x|^{b_1+\varepsilon}|u|\|_{L_{I_T}^{\infty}L_x^2}\||x|^{-(1-\varepsilon)}u\|_{L_{I_T}^{\infty}L_x^2}\|{\partial_\theta}u\|_{L_{I_T}^{\infty}L_x^{4}} +  \|\left|x\right|^{b_1}\left|u\right|\|_{L_{I_T}^{\infty}L_x^{\infty}}\|\partial u\|_{L_{I_T}^{\infty}L_x^{2}}\|{\partial_\theta}u\|_{L_{I_T}^{\infty}L_x^{4}})\\
&\qquad \lesssim T^{\frac{3}{4}}||u||_{L_{I_T}^{\infty}H_\theta^{1,1}}^3
\end{align*}
and
 \begin{align*}
 &\| \left|x\right|^{b_2-1}|u|^4\left|{\partial_\theta}u\right| \|_{L_{I_T}^{\frac{4}{3}}L_x^{\frac{4}{3}}} + \| \left|x\right|^{b_2}|u|^3|\partial u|\left|{\partial_\theta}u\right| \|_{L_{I_T}^{\frac{4}{3}}L_x^{\frac{4}{3}}}\\
 &\qquad \lesssim T^{\frac{3}{4}}(\| \left|x\right|^{\frac{b_2+\varepsilon}{3}}\left|u\right|\|_{L_{I_T}^{\infty}L_x^{r}}^3\||x|^{-(1-\varepsilon)} u \|_{L_{I_T}^{\infty}L_x^{2}} \|{\partial_\theta}u\|_{L_{I_T}^{\infty}L_x^{4}} + \| \left|x\right|^{\frac{b_2}{3}}\left|u\right| \|_{L_{I_T}^{\infty}L_x^{\infty}}^3 \|\partial u \|_{L_{I_T}^{\infty}L_x^{2}} \|{\partial_\theta}u\|_{L_{I_T}^{\infty}L_x^{4}})\\
 &\qquad \lesssim T^{\frac{3}{4}}||u||_{L_{I_T}^{\infty}H_\theta^{1,1}}^5.
  \end{align*}
 Therefore we get
\begin{align*}
\|\mathcal{H}(u)\|_{L_{I_T}^{\infty}H_\theta^{1,1}} \le \|\varphi\|_{H_\theta^{1,1}} + C\left(T+T^{\frac{3}{4}}\right)\left(\delta^3 + \delta^5\right).
\end{align*}
Set $\delta \ge 2||\varphi||_{H_\theta^{1,1}}$ and choose $T$ small enough that $C\left(T+T^{\frac{3}{4}}\right)\left(\delta^3 + \delta^5\right) \le \frac\delta2$. Then it can be shown that $\mathcal{H}$ the self-mapping on $X_\theta^{1,1}(T,\delta)$. Since the nonlinear terms are algebraic, we can carry out the estimates for $d^{1, 1}$ with a slight change of terms and a smaller $T$ as follows:
\begin{align*}
d^{1,1}\left(\mathcal{H}(u),\mathcal{H}(v) \right) \leq C(1+\|u\|_{L_{I_T}^{\infty}H_\theta^{1,1}} + \|v\|_{L_{I_T}^{\infty}H_\theta^{1,1}})^4 \|u-v\|_{L_{I_T}^{\infty}H_\theta^{1,1}} \leq \frac{1}{2}d^{1,1}(u,v).
\end{align*}
The local well-posedness in $H_\theta^{1, 1}$ is now clear from the contraction.

\subsubsection{Case: $b_1=0$ and $b_2=0$ }

For the proof we only consider the term $|x|^{-1}|u|^2|{\partial_\theta}u|$. Taking $(\widetilde q, \widetilde r) = (\frac83, 8)$ we have that
$$
\||x|^{-1}|u|^2|{\partial_\theta}u|\|_{L_{I_T}^{\frac85}L_x^{\frac87}} \lesssim T^{\frac58}\||x|^{-\frac23}u\|_{L_{t}^{\infty}L_x^2}^{\frac32}\|u\|_{L_{t}^{\infty}H_\theta^{1,1}}^{\frac12}\|{\partial_\theta}u\|_{L_{t}^{\infty}H^1}.
$$
Then we obtain
$$
\|N_1^{1,1}\|_{L_{I_T}^{\infty}L_x^2} \lesssim (T^{\frac38}+T^{\frac58})\|u\|_{H_\theta^{1,1}}^3 \lesssim (T^{\frac38}+T^{\frac58})\delta^3.
$$
By the same way, we estimate
$$
\|N_2^{1,1}\|_{L_{I_T}^{\infty}L_x^2} \lesssim (T^{\frac38}+T^{\frac58})\|u\|_{H_\theta^{1,1}}^5 \lesssim (T^{\frac38}+T^{\frac58})\delta^5.
$$
This gives us the local well-posedness.

\begin{rem}\label{lwp-in-h1}
We can treat the case $b_1 = b_2 = 0$ as usual energy-subcritical problem. We can show the local well-posedness in $H^1$. In fact, the main obstacle would be the terms $|x|^{-1}|u|^3$ and $|x|^{-1}|u|^5$ which can be estimated as follows:
\begin{align*}
\||x|^{-1}|u|^3\|_{L_{I_T}^{\frac43}L_x^{\frac43}} &\lesssim \||x|^{-\frac13}u\|_{L_{t}^4L_x^4}^3 \le \|D^\frac13 u\|_{L_T^4L_x^4}^3 \lesssim T^{\frac34}\|u\|_{L_t^\infty H^1}^3,\\
\||x|^{-1}|u|^5\|_{L_{I_T}^{\frac32}L_x^{\frac65}} &\lesssim T^\frac23\||x|^{-\frac15}u\|_{L_{t}^\infty L_x^6}^5 \le T^\frac23\|D^\frac15 u\|_{L_T^\infty L_x^6}^5 \lesssim T^{\frac23}\|u\|_{L_t^\infty H^1}^5.
\end{align*}
These give us the desired local well-posedness in $H^1$.
\end{rem}

\subsubsection{Case: $b_1 = \frac12$ and $b_2 = \frac32$  }
As in the proof of Section \ref{subs;1} we only consider the terms $|x|^{-\frac12}|u|^2|{\partial_\theta}u|$ and $|x|^{\frac12}|u||{\partial_\theta}u||\partial u|$.
Choose $(\widetilde q, \widetilde r) = (4, 4)$. Then for $0 < \varepsilon < \frac12$ we have from Lemma \ref{ang-sobo} that
\begin{align*}
&\||x|^{-\frac12}|u|^2|{\partial_\theta}u|\|_{L_{I_T}^{\frac43}L_{x}^{\frac43}} + \||x|^{\frac12}|u||{\partial_\theta}u||\partial u|\|_{L_{I_T}^{\frac43}L_{x}^{\frac43}}\\
&\qquad \lesssim T^{\frac34}( \||x|^{\frac12-\varepsilon}u\|_{L_T^\infty L_x^\infty}\||x|^{-1+\varepsilon}u\|_{L_T^\infty L_x^2}\|\partial_\theta u\|_{L_T^\infty L_x^4} +  \||x|^{\frac12}u\|_{L_t^{\infty}L_x^{\infty}}\|\partial u\|_{L_t^{\infty}L_x^2}\|{\partial_\theta}u\|_{L_t^{\infty}L_x^4})\\
& \lesssim T^{\frac34}\|u\|_{H_\theta^{1, 1}}\| u\|_{L_T^{\infty}H^1}\|{\partial_\theta}u\|_{L_T^{\infty}L_x^4}\\
& \lesssim T^{\frac34}\|u\|_{\hl^{1,1}}^3 \lesssim T^{\frac34}\delta^3.
\end{align*}
In particular, $||N_1^{1,1}||_{L_{I_T}^{\infty}L_x^2} \lesssim (T+T^{\frac34})\delta^3$. Similarly, we can get $||N_2^{1,1}||_{L_{I_T}^{\infty}L_x^2} \lesssim (T+T^{\frac34})\delta^5$.  This completes the proof of part (1) of Theorem \ref{thm;lwp}.

\subsubsection{Contraction on $X_\theta^{2, 1}(T, \delta)$}
Here we show that \eqref{eq:sch} is well-posed in $H^2$. We have only to estimate $N_1^2$ and $N_2^2$.
Let us choose $ 0 < \varepsilon < \frac{1}{4}\min(b_1, b_2)$. Then we first get
\begin{align*}
\|N_{1}^{2}\|_{L_{I_{T}}^{\infty}L_{x}^{2}} &\lesssim \|\left|x \right|^{b_1 -2}|u|^3\|_{L_{I_{T}}^{\frac{2}{1 + 2\varepsilon}}L_{x}^{\frac{1}{1-\varepsilon}}}
+ ||\left|x \right|^{b_1 -1}|u|^2\left|\partial u\right|||_{L_{I_{T}}^{\frac{4}{3}}L_{x}^{\frac{4}{3}}}\\
&\qquad + \|\left|x \right|^{b_1}|u|\left|\partial u\right|^2\|_{L_{I_{T}}^{1}L_{x}^{2}} + \|\left|x \right|^{b_1}|u|^2 \left|\partial^2 u\right|\|_{L_{I_{T}}^{1}L_{x}^{2}}\\
&\lesssim T^{\frac{1}{2}+\varepsilon} \|\left|x \right|^{\frac{b_1}{2(1+2\varepsilon)}}\left| u \right| \|_{L_{I_{T}}^{\infty} L_{x}^{\infty}}^{1+2\varepsilon}\|\left|x \right|^{-\frac{1-\frac{b_1}{4}}{1-\varepsilon}}\left|u\right|\|_{L_{I_{T}}^{\infty}L_{x}^{2}}^{2(1-\varepsilon)}\\
&\qquad + T^{\frac{3}{4}} \|\left|x \right|^{\frac{b_1}{2}}\left|u\right| \|_{L_{I_{T}}^{\infty}L_{x}^{\infty}}\|\left|x \right|^{-(1- \frac{b_1}{2})}\left| u\right| \|_{L_{I_{T}}^{\infty}L_{x}^{2}}\|\partial u\|_{L_{I_{T}}^{\infty}L_{x}^{4}}\\
&\qquad + T \left( \|\left|x \right|^{b_1}\left| u\right| \|_{L_{I_{T}}^{\infty}L_{x}^{\infty}}\|\partial u\|_{L_{I_{T}}^{\infty}L_{x}^{4}}^{2}
+ \|\left|x \right|^{\frac{b_1}{2}}\left| u\right| \|_{L_{I_{T}}^{\infty}L_{x}^{\infty}}^{2}\|\partial^{2} u \|_{L_{I_{T}}^{\infty}L_{x}^{2}}\right)\\
&\lesssim T^{\frac{1}{2}+\varepsilon} \|u \|_{L_{I_{T}}^{\infty} H_{\theta}^{1,1}}^{1+2\varepsilon}\|u\|_{L_{I_{T}}^{\infty}L_{x}^{2}}^{2(1-\varepsilon)}
+ T^{\frac{3}{4}} \|u \|_{L_{I_{T}}^{\infty}H_{\theta}^{1,1}}^{2}\|u \|_{L_{I_{T}}^{\infty}H^{2}}\\
&\qquad +T \left( \|u \|_{L_{I_{T}}^{\infty}H_{\theta}^{1,1}}\|u \|_{L_{I_{T}}^{\infty}H^{2}}^{2}
+ \|u \|_{L_{I_{T}}^{\infty}H_{\theta}^{1,1}}^{2}\|u \|_{L_{I_{T}}^{\infty}H^{2}} \right)\\
&\lesssim \left( T^{\frac{1}{2}+\varepsilon} + T^{\frac{3}{4}}+T \right) \delta ^3.
\end{align*}
By the same way we can estimate
\begin{align*}
||N_{2}^{2}||_{L_{I_{T}}^{\infty}L_{x}^{2}} \lesssim \left( T^{\frac{1}{2}+\varepsilon} + T^{\frac{3}{4}}+T \right) \delta ^5.
\end{align*}
This concludes the proof of local well-posedness in $H^2$.

\subsubsection{Mass and Energy conservation}
We have just shown that for any $\varphi\in H_\theta^{2,1}$ the solution $u\in C(I_*;H_\theta^{2,1})$. So we first assume that $\varphi\in H_\theta^{2,1}$. Then the map $g(u) = K_1Q_1(u)+K_2Q_2(u)\in C(H_\theta^{2,1},L_x^2)$. Hence if $u\in C(I_T;H_\theta^{2,1})$ for any $I_T \subset I_*$, then $u_t\in C(I_T;L_x^2)$. This implies the mass and energy conservation. For this see \cite{caze}. Now, using the continuous dependency of solution on the initial data, the mass and energy conservation in $H_\theta^{1,1}$ follows.

\begin{rem}\label{h1gwp}
If $b_1 = b_2 = 0$ and $K_l \ge 0$ (defocusing), then by energy conservation leads us to the uniform boundedness of $H^1$ norm of solutions and global well-posedness.
\end{rem}

\subsection{Blowup criterion}
We will show the finite time blowup via localized virial argument (\cite{mer2}).

\begin{lem}\label{loc-virial}
Let $\varphi \in  H_\theta^{1, 1}$ and $x\varphi \in L_{x}^{2}$, and let $u$ be the solution of (\ref{eq:sch}) in $C( [-T,T];H_{\theta}^{1,1} )$. Then $xu \in C( [-T,T];L_{x}^{2} )$ and it satisfies that
\begin{eqnarray}\label{dilation}
\|\left|x\right|u(t)\|_{L_{x}^{2}}^{2} = \|\left|x\right|\varphi \|_{L_{x}^{2}}^{2} + 4 \int_0^t\mathcal{A}(t)dt,
	\end{eqnarray}
where $\mathcal{A}(t)= {\rm Im} \int \bar{u}(t) \left(x\cdot \nabla\right)u(t)dx$.
\begin{eqnarray}\label{virial}
  \frac{d}{dt}\mathcal{A}(t)= 4E\left( u(t)\right) -\frac12\int(x \cdot \nabla K_1)|u|^4 dx + \frac13\int \left(2K_2 - x\cdot \nabla K_2 \right)|u|^6 dx.
\end{eqnarray}
\end{lem}

\begin{proof} By the continuous dependency on initial data we may assume that $H_\theta^{2, 1}$.
Let $a(x)\in C_0^\infty(\mathbb{R}^2)$  and $a_r(x)$ be as follows:
\begin{align*}a(x)&:=\left\{ \begin{array}{cc}
                  |x|^2 & (|x|\leq 1) \\
                  0 & (|x|\geq 2)
                \end{array}  \right.,\quad
a_r(x):=r^2a(\frac{x}{r}).
\end{align*}
Set $D_r(t) = \int a_r|u(t)|^2 dx$. Then by direct calculation we have
$$
\frac{d}{dt}D_r = -2 \textrm{Im}\int a_r \Delta u \bar{u}\,dx
$$
and
\begin{align*}
\frac{d^2}{dt^2}D_r &= 2 \textrm{Im}\int \left[-\Delta a_r u_t \bar{u}-\left(\nabla a_r \cdot \nabla\bar{u}\right)u_t +\left(\nabla a_r \cdot \nabla u \right)\bar{u_t}\right] dx\\
&= 2 \int \Delta a_r|\nabla u|^2 dx - \int \Delta^2a_r|u|^2 dx +2 \int\Delta a_r\left(K_1|u|^4+K_2|u|^6\right)dx\\
&\qquad + \textrm{Re}\int(\nabla^2 a_r \cdot \nabla\bar{u})\nabla u dx -\frac12 \int \Delta a_r |\nabla u|^2dx\\
&\qquad - \int\Delta a_r K_1|u|^4 dx - \frac23\int\Delta a_r K_2|u|^6 dx - \int\left(\nabla a_r \cdot \nabla K_1\right)|u|^4 dx -\frac23\int\left(\nabla a_r \cdot \nabla K_2\right)|u|^6 dx\\
&= \int_{|x|\leq r}U(x)dx + \int_{r \leq |x| \leq 2r} U(x)dx,
\end{align*}
where
\begin{align*}
U(x) &= 2\Delta a_r|\nabla u|^2  -  \Delta^2a_r|u|^2  +2 \Delta a_r\left(K_1|u|^4+K_2|u|^6\right)+  \textrm{Re}(\nabla^2 a_r \cdot \nabla\bar{u})\nabla u  -\frac12 \Delta a_r |\nabla u|^2\\
&\qquad - \Delta a_r K_1|u|^4 dx - \frac23\Delta a_r K_2|u|^6  - \left(\nabla a_r \cdot \nabla K_1\right)|u|^4 - \frac23\left(\nabla a_r \cdot \nabla K_2\right)|u|^6.
\end{align*}

Since $\int_{r \leq |x| \leq 2r}U(x)dx \rightarrow 0$ as $r \rightarrow \infty$, after integrating over $[0, t]$, by taking limit $r \rightarrow \infty$ and then derivatives, we obtain \eqref{dilation} and \eqref{virial}.
\end{proof}

Now from Lemma \ref{loc-virial} it follows that
\begin{align*}
&\frac{d^2}{dt^2}\||x|u(t)\|_{L_x^2}^{2} \\
&\quad = 8 \int|\nabla u|^2 dx + 4 \int  K_1|u|^4 dx + \frac{16}3\int K_2|u|^6 dx - 2\int\left(x \cdot \nabla K_1\right)|u|^4 dx -\frac43\left(x \cdot \nabla K_2\right)|u|^6 dx\\
&\quad = 16E\left(u(t)\right) -2 \int\left(x \cdot \nabla K_1\right)|u|^4dx + \frac43\int\left(2K_2 -x \cdot \nabla K_2 \right)|u|^6dx\\
&\quad \leq 16E\left(u(t)\right) + 8\alpha\left(\frac14 \int K_1|u|^4dx + \frac16\int K_2|u|^6 dx\right)\\
&\quad \leq (16 + 8\alpha)E(\varphi).
\end{align*}
In particular,
\begin{eqnarray*}
\||x|u(t)\|_{L_x^2}^{2} \le (8 + 4\alpha)t^2 E(\varphi)+ 4t \textrm{Im} \int\left(x\cdot\nabla\varphi\right)dx+\int|x|^2|\varphi|^2 dx.
\end{eqnarray*}
Since $E(\varphi)< 0$, the last inequality gives us the finite time blowup.

\subsection{Radial global well-posedness}
If $K_l \ge 0$, then \eqref{eq:sch} is defocusing. The uniform bound of $\|u\|_{H_{rad}^1}$ follows from the energy conservation and then we get the global well-posedness.

Let us consider the case of no sign condition for $K_l$. For this we need the smallness of $H_{rad}^1$ norm because \eqref{eq:sch} has mass-supercritical but energy-subcritical nature due to the quintic term.
Under radial symmetry the following holds from an interpolation.
\begin{lem}[Proposition 3 of \cite{soin2}]\label{radial-soin} Let $ 0 < b \le \frac12$. Then for any $f \in H_{rad}^1$ we have
$$
\||x|^{b}f\|_{L_x^{\infty}} \lesssim \|f\|_{L_x^2}^b\|\nabla f\|_{L_x^2}^{1-b}.
$$
\end{lem}
At first we consider the case $b_1 > 0$. Using Lemma \ref{radial-soin} and Gagliardo-Nirenberg's inequality, we have
\begin{align*}
\int K_1|u|^4\,dx &\le \||x|^\frac{b_1}2u\|_{L_x^\infty}^2\|u\|_{L_x^2}^2 \lesssim \|u\|_{L_x^2}^{b_1}\|\nabla u\|_{L_x^2}^{2-b_1}\|u\|_{L_x^2}^2\\
&\lesssim \|u\|_{L_x^2}^{2+b_1}\|\nabla u\|_{L_x^2}^{2-b_1} \le m(\varphi)^{1 + \frac{b_1}2} \|\nabla u\|_{L_x^2}^{2-b_1},
\end{align*}
and
\begin{align*}
\int K_2|u|^6\,dx & \lesssim \left\{\begin{array}{l}\||x|^\frac{b_2}4u\|_{L_x^\infty}^4\|u\|_{L_x^2}^2 \le m(\varphi)^{1 + \frac{b_2}2} \|\nabla u\|_{L_x^2}^{4-b_1}, \;\;\mbox{if}\;\;b_2 > 0,\\
m(\varphi) \|\nabla u\|_{L_x^2}^{4},\;\;\mbox{if}\;\; b_2 = 0.\end{array} \right.
\end{align*}
Then Young's inequality gives us that for some $\delta > 0$
\begin{align*}
E(\varphi) &\ge \frac12\|\nabla u\|_{L_x^2}^2 - \frac14\int |K_1||u|^4\,dx - \frac16\int|K_2||u|^6\,dx \\
&\ge \frac14\|\nabla u\|_{L_x^2}^2 - Cm(\varphi)^\frac{2 + b_1}{b_1} -  m(\varphi)^{1+ \frac{b_2}2}\|\nabla u\|_{L_x^2}^{4-b_1}.
\end{align*}
Therefore we get the uniform bound of $\|\nabla u\|_{L_x^2}^2$ by continuity argument, provided $\|\nabla \varphi\|_{L_x^2}$ is sufficiently small.

If $b_1 = 0$, then since $\int |K_1||u|^2\,dx \lesssim \|u\|_{L_x^2}^2\|\nabla u\|_{L_x^2}^2$, we need smallness of $\|\varphi\|_{L_x^2}$.

\section{Small data scattering}
In order to show the scattering we use the extended Strichartz estimates of Lemma \ref{lem;stri-extended}.
To begin with we introduce a Besov type function space $B_{\theta, r, \widetilde r}^s$ for $s \in \mathbb R$ and $1 \le r, \widetilde r \le \infty$, which is defined by
$$
\left\{f \in L_\rho^rL_\theta^{\widetilde r} : \|f\|_{B_{\theta}, r, \widetilde r }^s := \left(\sum_{N \ge 1} N^{2s}\|P_Nf\|_{L_\rho^rL_\theta^{\widetilde r}}^2\right)^\frac12 < \infty \right\},
$$
where $P_N$ is the frequency projection operator for dyadic numbers $N \ge 1$ such that $\sum_{N \ge 1}P_N = 1$, $\widehat{P_1 f} = \beta_1\widehat f$, and $\widehat{P_Nf} = \beta(\frac{\cdot}{N})\widehat f$ for usual Littlewood-Paley functions $\beta_1 \in C_0^\infty(B(0, 1))$ and $\beta \in C_0^\infty (\frac12 < |\xi| < 2)$.
Then we can easily get the following:
\begin{enumerate}
\item $B_{\theta, r, \widetilde r}^s$ is a Banach space whose norm is $\|\cdot\|_{B_{\theta, r, \widetilde r} }^s$.
\item Its dual is $B_{\theta, r', \widetilde r'}^{-s}$.
\item $\|f\|_{B_{\theta, r, \widetilde r}^{1+s}} \sim \|f\|_{B_{\theta, r, \widetilde r}^{s}} + \|\nabla f\|_{B_{\theta, r, \widetilde r}^{s}}$ for $1 < r, \widetilde r < \infty$.
\item It has natural real and complex interpolation structure.
\end{enumerate}
For these see \cite{tr, chonak}.

Now let us define a set $\triangle$ of extended Strichartz pairs $(q, r)$  by
$$
\triangle := \{(\infty, 2)\} \cup \left\{(q, r) : \frac12 - \frac1r < \frac1q < \frac32(\frac12 - \frac1r), 2 \le q < \infty < \infty, 2 < r < \infty\right\}.
$$
For any pair $(q, r) \in \triangle$ set $s(q,r) = 2(\frac1q+\frac1r - \frac12)$. Then from Lemma \ref{lem;stri-extended}, Littlewood-Paley theory, and complex interpolation it follows that
\begin{align}\label{str-besov1}
\|e^{it\Delta}\varphi\|_{L_t^qB_{\theta, r, 2}^{s(q, r)}} \lesssim \|\varphi\|_{L_x^2}.
\end{align}
It can be shown by Christ-Kiselev lemma (for instance see \cite{chrkis, ahch}) and duality argument that for any $(q, r), (\widetilde q, \widetilde r) \in \triangle$ with $\widetilde q' < q$
\begin{align}\label{str-besov2}
\left\|\int_0^t e^{i(t-t')\Delta}F(t')\,dt'\right\|_{L_t^qB_{\theta, r, 2}^{s(q, r)}} \lesssim \|F\|_{L_t^{\widetilde q'}B_{\theta, \widetilde r', 2}^{-s(\widetilde q, \widetilde r)}},
\end{align}
and also that for any $(q, r) \in \triangle$ and any admissible pair $(\widetilde q, \widetilde r)$ with $\widetilde q' < q$
\begin{align}\label{str-besov3}
\left\|\int_0^t e^{i(t-t')\Delta}F(t')\,dt'\right\|_{L_t^qB_{\theta, r, 2}^{s(q, r)}} \lesssim \|F\|_{L_t^{\widetilde q'}L_x^{\widetilde r'}}.
\end{align}
In this paper we only use \eqref{str-besov1} and \eqref{str-besov3}.

Let us choose $6 < r$ such that $\max(b_1, b_2 - 1) \le \frac2r$. Then we define the complete metric space $\xl(\delta)$ with metric $d$ by
\begin{align*}
X_{\theta}\left(\delta \right) := \{ u\in ( C\cap L_t^{\infty})(\mathbb{R};H_{\theta}^{1,1} ) : \|u\|_{X_\theta} &:= \|u \|_{L_t^{\infty}H_{\theta}^{1,1}} + \|u \|_{L_t^{4}H_{\theta, 4}^{1,1}} + \|u \|_{L_t^{3}H_{\theta, 6}^{1,1}} + \|(u, \partial_\theta u)\|_{L_t^{2}B_{\theta, r, 2}^{1+s(2, r)}} \leq \delta \},\\
 d(u,v) &:= \|u-v \|_{X_\theta}.
\end{align*}

Now we show that the nonlinear functional $\mathcal H(u) = e^{it\Delta}\varphi + N(u)$ is a contraction on $\xl(\delta)$.
For this we have only to show
\begin{align}
&\| N(u)\|_{\xl} \lesssim \|u\|_{\xl}^3 + \|u\|_{\xl}^5,\label{contract1}\\
&\| N(u) - N(u)\|_{\xl} \lesssim [(\|u\|_{\xl} + \|v\|_{\xl})^2 + (\|u\|_{\xl} + \|v\|_{\xl})^4]\|u-v\|_{\xl}.\label{contract2}
\end{align}
Indeed, $\|e^{it\Delta} \varphi\|_{\xl} \lesssim \|\varphi\|_{H_\theta^{1, 1}}$ by the extended Strichartz estimate \eqref{str-besov1}. From \eqref{contract1} and \eqref{contract2} we can find $\delta$ small enough for $\mathcal H$ to be a contraction mapping on $\xl(\delta)$, and for the equation (\ref{eq:sch}) to be globally well-posed in $H_\theta^{1, 1}$.

Since the nonlinear terms are algebraic (cubic and quintic), \eqref{contract2} follows from \eqref{contract1} straightforwardly. So, we consider only \eqref{contract1}. We utilize Lemmas \ref{lem;soin} -- \ref{lem;cm}, and Strichartz estimate \eqref{str-besov3}.

Let us invoke $N_l^{1,1}$ in Section \ref{sec;lwp}. Then we further decompose them into two parts, inside and outside of the unit ball.
$$
N_l^{1, 1} = \sum_{l, k = 1, 2}\int_0^te^{i(t-t')\Delta} [\psi_k\partial\partial_\theta Q_l]\,dt',
$$
where $\psi_1 \in C_0^\infty(B(0 , 1))$ and $\psi_2 = 1 - \psi_1$.
Given $\delta$, taking $(\widetilde q, \widetilde r) = (4, 4)$ or $(\infty, 2)$, we have that for any $u\in \xl(\delta)$
\begin{align*}
\|N_{l}^{1,1}\|_{L_{t}^{\infty}L_{x}^{2} \cap L_t^4L_x^4 \cap L_t^3L_x^6 \cap L_t^2B_{\theta, r, 2}^{s(2, r)}} &\lesssim \sum_{j = 1}^5 Q_{l, j}^k, \;\;l, k = 1, 2,
\end{align*}
where
\begin{align*}
Q_{l, 1}^1 &= \|\psi_1|x|^{b_l-1}|u|^{2l+1}\|_{L_t^\frac43L_x^\frac43},\qquad Q_{l, 2}^1 = \|\psi_1 |x|^{b_l}|u|^{2l}|\partial u|\|_{L_t^\frac43L_x^\frac43}, \qquad Q_{l, 3}^1 = \|\psi_1|x|^{b_l-1}|u|^{2l}|\partial_\theta u|\|_{L_t^\frac43L_x^\frac43},\\
Q_{l, 4}^1 &= \|\psi_1|x|^{b_l}|u|^{2l-1}|\partial u||\partial_\theta u|\|_{L_t^\frac43L_x^\frac43}, \qquad Q_{l, 5}^1 = \|\psi_1|x|^{b_l}|u|^{2l}|\partial \partial_\theta u |\|_{L_t^\frac43L_x^\frac43},
\end{align*}
and
\begin{align*}
Q_{l, 1}^2 &= \|\psi_2|x|^{b_l-1}|u|^{2l+1}\|_{L_t^\frac43L_x^\frac43},\qquad Q_{l, 2}^2 = \|\psi_2 |x|^{b_l}|u|^{2l}|\partial u|\|_{L_t^1L_x^2}, \qquad Q_{l, 3}^2 = \|\psi_2|x|^{b_l-1}|u|^{2l}|\partial_\theta u|\|_{L_t^\frac43 L_x^\frac43},\\
Q_{l, 4}^2 &= \|\psi_2|x|^{b_l}|u|^{2l-1}|\partial u||\partial_\theta u|\|_{L_t^1L_x^2}, \qquad Q_{l, 5}^2 = \|\psi_2|x|^{b_l}|u|^{2l}|\partial \partial_\theta u |\|_{L_t^1L_x^2}.
\end{align*}

By H\"older's and Hardy-Sobolev's inequalities, and Lemma \ref{lem;ang-soin} we estimate $Q_{l, j}^1$ with $0 < \varepsilon < \min(\frac12, 1 - (b_2-[b_2]))$ as follows:
\begin{align*}
Q_{1, 1}^1 + Q_{2, 1}^1 &\lesssim \||x|^{-\frac{1-b_1}3}u\|_{L_t^4 L_x^4}^3 + \|\psi_1|x|^{\varepsilon}|u|^2\|_{L_t^\infty L_x^\infty}\||x|^{-\frac{1-(b_2-[b_2]+\varepsilon)}3}u\|_{L_{t}^4L_x^4}^3\\
& \lesssim \|u\|_{L_t^4H_{4}^1}^3 + \|u\|_{L_t^\infty H_\theta^{1, 1}}^{2}\|u\|_{L_t^4H_4^1}^3 \lesssim \delta^3 + \delta^5,\qquad\qquad\qquad\qquad\qquad\qquad\qquad\qquad\quad\;\;\;
\end{align*}
\begin{align*}
Q_{1, 2}^1 + Q_{2, 2}^1 &\lesssim \|u\|_{L_t^4L_x^4}^3 + \||x|^{\varepsilon}u\|_{L_t^\infty L_x^\infty}^{2}\||x|^{-\varepsilon}u\|_{L_t^4L_x^4}^2\|\partial u\|_{L_t^4L_x^4} \\
&\lesssim \|u\|_{L_t^4H_4^1}^3 + \|u\|_{L_t^\infty H_\theta^{1,1}}^2\|u\|_{L_t^4H_4^1}^3 \lesssim \delta^3 + \delta^5,\qquad\qquad\qquad\qquad\qquad\qquad\qquad\qquad\quad\quad
\end{align*}
\begin{align*}
Q_{1, 3}^1 + Q_{2, 3}^1 &\lesssim \||x|^{-\frac{1-b_1}3}u\|_{L_t^4 L_x^4}^2\||x|^{-\frac{1-b_1}3}\partial_\theta u\|_{L_t^4 L_x^4}\\
&\qquad + \|\psi_1|x|^{\varepsilon}|u|^2\|_{L_t^\infty L_x^\infty}\||x|^{-\frac{1-(b_2-[b_2]+\varepsilon)}3}u\|_{L_{t}^4L_x^4}^2\||x|^{-\frac{1-(b_2-[b_2]+\varepsilon)}3}\partial_\theta u\|_{L_{t}^4L_x^4}\\
& \lesssim \|u\|_{L_t^4H_{4}^1}\|u\|_{L_t^4H_{\theta, 4}^{1,1}}^2 + \|u\|_{L_t^\infty H_\theta^{1, 1}}^{2}\|u\|_{L_t^4H_{\theta, 4}^{1,1}}^3 \lesssim \delta^3 + \delta^5,\qquad\qquad\qquad\qquad\qquad\qquad\quad
\end{align*}
\begin{align*}
Q_{1, 4}^1 + Q_{2, 4}^1 &\lesssim \|u\|_{L_t^4L_x^4}\|\partial u\|_{L_t^4L_x^4}\|\partial_\theta u\|_{L_t^4L_x^4} + \||x|^{2\varepsilon}|u|^2\|_{L_t^\infty L_x^\infty}\||x|^{-2\varepsilon}u\|_{L_t^4L_x^4}\|\partial u\|_{L_t^4L_x^4}\|\partial_\theta u\|_{L_t^4L_x^4}\\
& \lesssim \|u\|_{L_t^4H_{\theta, 4}^{1,1}}^3 + \|u\|_{L_t^\infty H_\theta^{1,1}}^2\|u\|_{L_t^4H_{\theta, 4}^{1,1}}^3 \lesssim \delta^3 + \delta^5,
\end{align*}
\begin{align*}
Q_{1, 5}^1 + Q_{2, 5}^1 &\lesssim \|u\|_{L_t^4L_x^4}^2\|\partial\partial_\theta u\|_{L_t^4L_x^4} + \||x|^{2\varepsilon}|u|^2\|_{L_t^\infty L_x^\infty}\||x|^{-\varepsilon}u\|_{L_t^4L_x^4}^2\|\partial\partial_\theta u\|_{L_t^4L_x^4}\\
& \lesssim \|u\|_{L_t^4H_{\theta, 4}^{1,1}}^3 + \|u\|_{L_t^\infty H_\theta^{1, 1}}^{2}\|u\|_{L_t^4H_4^1}^2\|u\|_{L_t^4H_{\theta, 4}^{1, 1}} \lesssim \delta^3 + \delta^5.\qquad\qquad\qquad\qquad\qquad\qquad\;\;
\end{align*}
Here $[b_2]$ denotes the maximal integer less than or equal to $b_2$.

Next we estimate $Q_{l, j}^2$. We first have
\begin{align*}
Q_{1, 1}^2\le \||x|^{-\frac{1-b_1}3}u\|_{L_t^4L_x^4}^3  \lesssim \|u\|_{L_t^4H_4^1}^3 \lesssim \rho^3.
\end{align*}
If $[b_2] = 0$, then by choosing a small $\varepsilon < \frac14$ we have
\begin{align*}
Q_{2, 1}^2 \le \||x|^{2\varepsilon}|u|^2\|_{L_t^\infty L_x^\infty}\||x|^{-\frac{1+2\varepsilon}3}u\|_{L_t^4L_x^4}^3  \lesssim \|u\|_{L_t^\infty H_\theta^{1,1}}^2\|u\|_{L_t^4H_4^1}^3 \lesssim \rho^5.
\end{align*}
If $[b_2] = 1$, then
\begin{align*}
Q_{2, 1}^2 \le \||x||u|^2\|_{L_t^\infty L_x^\infty}\||x|^{-\frac{1-(b_2 - [b_2])}3}u\|_{L_t^4L_x^4}^3  \lesssim \|u\|_{L_t^\infty H_\theta^{1,1}}^2\|u\|_{L_t^4H_4^1}^3 \lesssim \rho^5.
\end{align*}
As for $Q_{l, 2}^2$ we estimate:
If $b_1 = 0$, then
\begin{align*}
Q_{1, 2}^2\le \|u\|_{L_t^3 L_x^6}^2\|u\|_{L_t^3H_6^1}  \lesssim \rho^3.
\end{align*}
If $0 < b_1 < \frac13$, then since $b_1 \le  \frac2r$, using Corollary \ref{decay}, we obtain
\begin{align}\begin{aligned}\label{q122}
Q_{1, 2}^2 &\le \||x|^{\frac2r}u\|_{L_t^\infty L_\rho^\frac{2r}{r-4} L_\theta^\infty}\|u\|_{L_t^2L_\rho^rL_\theta^2}\|\partial u\|_{L_t^2L_\rho^rL_\theta^\infty} \lesssim  \|u\|_{L_t^\infty H_\theta^{1,1}}\|u\|_{L_t^2L_\rho^rL_\theta^2}\|(\partial u, \partial_\theta\partial u)\|_{L_t^2L_\rho^rL_\theta^2}\\
&\lesssim \|u\|_{L_t^\infty H_\theta^{1,1}}\|(u, \partial_\theta u)\|_{L_t^2B_{\theta, r, 2}^{1+s(2, r)}}^2 \lesssim \rho^3.
\end{aligned}\end{align}
Here we used the fact $\|\partial_\theta \partial f\|_{L_\rho^rL_\theta^2} \lesssim \|\nabla f\|_{L_\rho^rL_\theta^2} + \|\partial \partial_\theta f\|_{L_\rho^2L_\theta^2}$.

If $[b_2] = 0$, then we choose a small $\varepsilon$ such that $\varepsilon < \frac13$ and $b_2 +2\varepsilon < 1$ and get
\begin{align*}
Q_{2,2}^2 &\le \||x|^\frac{b_2 + 2\varepsilon}2u\|_{L_t^\infty L_x^\infty}^2\||x|^{-\varepsilon}u\|_{L_t^3L_x^6}^2\|\partial u\|_{L_t^3L_x^6} \lesssim \|u\|_{L_t^\infty H_\theta^{1,1}}^2\|u\|_{L_t^3H_6^1}^3 \lesssim \rho^5.
\end{align*}
If $[b_2] = 1$, then by replacing $b_1$ in \eqref{q122} with $b_2-1$ we have
\begin{align*}
Q_{2,2}^2 &\le \||x|^\frac12u\|_{L_t^\infty L_x^\infty}^2\|\psi_2|x|^{b_2-1}|u|^2|\partial u|\|_{L_t^1L_x^2}
 \lesssim \|u\|_{L_t^\infty H_\theta^{1,1}}^3\|(u, \partial_\theta u)\|_{L_t^2B_{\theta, r, 2}^{1+s(2, r)}}^2 \lesssim \rho^5.
\end{align*}

Then we estimate $Q_{l, 3}^2$ and $Q_{l, 4}^2$ by replacing a $u$ in $Q_{l, 1}^2$ and $Q_{l, 2}^2$ with $\partial_\theta u$, respectively and obtain that
$$
\sum_{l = 1, 2}(Q_{l, 3}^2 + Q_{l, 4}^2) \lesssim \rho^3 + \rho^5.
$$
The estimate for $Q_{l, 5}^2$ can be done by replacing $\partial u$ of $Q_{l, 2}^2$ with $\partial \partial_\theta u$.
These conclude the estimate \eqref{contract1} and hence \eqref{contract2}.

The small data scattering is now straightforward from the global well-posedness. In fact, let us define a scattering state $u_\pm$ with
$$
\varphi_\pm := \varphi + \lim_{t\to \pm\infty} e^{-it\Delta}N(u).
$$
The existence of limit is guaranteed by the global well-posedness. Then we get the desired result by the duality argument based on the nonlinear estimates for $Q_{l, j}^k$:
\begin{align*}
\|u(t) - u_\pm(t)\|_{H_\theta^{1,1}} &= \|(u(t) - u_\pm(t), \partial_\theta(u(t) - u_\pm(t))\|_{H^1} \\
&= \sup_{\|\psi\|_{L^2} \le 1}\left|\int_t^{\pm \infty}\langle (1-\Delta)^\frac12(N(u), \partial_\theta N(u)), e^{-it'\Delta}\psi \rangle \,dt'\right|\\
& \lesssim (\delta^{2} + \delta^4)\|(u, \partial_\theta u)\|_{L^2((t, \pm\infty); B_{\theta, r, 2}^{1+s(2, r)}) \cap  L^3((t, \pm\infty); H_6^1) \cap L^4((t, \pm\infty); H_{4}^{1})} \to 0\quad \mbox{as}\quad t \to \pm \infty.
\end{align*}
Here $(t, \pm\infty)$ means that $(t, +\infty)$ if $t > 0$ and $(-\infty, t)$ if $t < 0$.
This completes the proof of Theorem \ref{thm;scattering}.

\section{Non-existence of scattering}
We follow the argument as in \cite{barab, 3d, choz}. By contradiction we assume that $\|\varphi_+\|_{L_x^2} \neq 0$. Since $K_l$ are real-valued, $m(u(t)) = m(\varphi)$.
We consider $H(t) = -{\rm Im}\int u(t)\overline{u_+(t)}\,dx$ for $t \gg 1$. Differentiating $H$, we get
$$
\frac{d}{dt}H(t) = {\rm Re}\int (K_1Q_1 + K_2Q_2) \overline{u_+}\,dx.
$$
We decompose this as follows:
$$
\frac{d}{dt}H(t) = \sum_{j = 1}^3J_1^j +  J_2,
$$
where
\begin{align*}
J_1^1 &= \int |x|^{b_1}|u_+|^4\,dx,\\
J_1^2 &= \int |x|^{b_1}(|u|^2-|u_+|^2)|u_+|^2\,dx,\\
J_1^3 &= {\rm Re}\int |x|^{b_1}|u|^2(u-u_+) \overline{u_+}\,dx,
\end{align*}
and
\begin{align*}
J_2 = {\rm Re}\int |x|^{b_2}|u|^4uu_+\,dx.
\end{align*}

We estimate $J_1^1$ as follows: for $0 < \delta \ll 1 \ll k$
$$
\int_{\delta t \le |x| \le kt}|u_+|^2\,dx \le \||x|^{-\frac{b_1}2}\|_{L_x^2(\delta t \le |x| \le kt)}(J_1^1)^\frac12
\lesssim t^{1-\frac{b_1}2}(J_1^1)^\frac12.
$$

It was show in \cite{barab, choz} that $\int_{\delta t \le |x| \le kt}|u_+|^2\,dx \sim \|\varphi_+\|_{L^2}^2$ for some fixed large $k$ and small $\delta$, and for any large $t$.
From this we deduce that $$J_1^1 \gtrsim_{m(\varphi_+)} t^{-(2-b_1)}.$$

Let us denote the generator of Galilean transformation by $\mathbf J$, that is $\mathbf J = e^{-it\Delta}x e^{it\Delta}$. On the sufficiently regular function space
\begin{align}\label{j}
\mathbf J = x + 2it\nabla, \quad (\mathbf J \cdot \mathbf J)^m = (|x|^2 - 4tA - 4t^2\Delta)^m,
\end{align}
where $A$ is the self-adjoint dilation operator defined by $\frac1{2i}(x\cdot \nabla + \nabla \cdot x)$, which  yields $\mathcal A = \int \overline{u}Au\,dx$.
Since $\|u_+(t)\|_{L_x^\infty} \lesssim t^{-1}\|\varphi_+\|_{L_x^1}$,
and
$$
\||x|^{2m} u_+(t)\|_{L_x^\infty} = \|e^{it\Delta}(\mathbf J \cdot \mathbf J)^{m}\varphi_+\|_{L_x^\infty} \lesssim_{\varphi_+}  t^{-(1-2m)},
$$
by interpolation we see that
\begin{align}\label{t-decay}
\||x|^\theta u_+(t)\|_{L_x^\infty} \lesssim_{\varphi_+} t^{-(1-\theta)}
\end{align}
for any $\theta > 0$.
By this we get $\||x|^{-\frac{b_1}3}u_+(t)\|_{L_x^\frac32(\delta t \le |x| \le kt)}(J_1^2)^\frac13  \lesssim t^{\frac43-\frac{b_1}3}(J_1^2)^\frac13$

For $J_1^2$ we have
\begin{align*}
J_1^2 &\lesssim \||x|^{b_1}|u_+(t)|^2\|_{L_x^\infty}(\|u\|_{L_x^2} + \|u_+\|_{L_x^2})\|u - u_+\|_{L_x^2}.
\end{align*}
Using \eqref{t-decay} we get $$|J_1^2| = o_{m(\varphi),\; \varphi_+}(t^{-(2-b_1)}).$$

To estimate $J_1^3$ and $J_2$ we need the following lemma.
\begin{lem}\label{pot-decay}
	Let $u$ be a global smooth solution of \eqref{eq:sch} with $K_1 = |x|^{b_1}, K_2 = |x|^{b_2}$ such that $b_1 > 0$ and $0 \le b_2 \le 2 +b_1$. If $xu \in C(\mathbb R; L_x^2)$, then for any large $t$ there holds
	$$
	V(u) := \frac14\int |x|^{b_1}|u|^4\,dx + \frac16\int |x|^{b_2}|u|^6\,dx \le C(m(\varphi), E(\varphi), \||x|\varphi\|_{L_x^2})\; t^{-(2-b_1)}.
	$$
\end{lem}
From Lemma \ref{pot-decay} and inequality \eqref{t-decay} it follows that
$$
|J_1^3| \le \left(\int|x|^{b_1}|u|^4\,dx\right)^\frac12\|u - u_+\|_{L_x^2}\||x|^\frac{b_1}2u_+\|_{L_x^\infty} = o_{\varphi, \; \varphi_+}(t^{-(2-b_1)}),
$$
\begin{align*}
|J_2| &\le  \int |x|^{\frac{5b_2}6}|u|^5|x|^\frac{b_2}6|u_+|\,dx = (\int |x|^{b_2}|u|^6\,dx)^\frac56\||x|^\frac{b_2}4|u_+|\|_{L_x^\infty}^\frac23\|u_+\|_{L_x^2}^\frac13 \lesssim_{\varphi,\; \varphi_+}t^{-(\frac73 - \frac{5b_1}6-\frac{b_2}6)}\\
& =  o_{\varphi, \;\varphi_+}(t^{-(2-b_1)}) \quad(\because b_2 < 2+b_1).
\end{align*}
Therefore we conclude that for $t \gg 1$
$$
\frac{d}{dt}H(t) \gtrsim_{\varphi, \varphi_+} t^{-(2-b_1)}.
$$
Since $H(t)$ is uniformly bounded for any $t \ge 0$, the range $b_1 \ge 1$ leads us to the contradiction to the assumption $\|\varphi_+\|_{L_x^2} \neq 0$. By time symmetry a similar argument holds for negative time. We omit that part.

\begin{proof}[Proof of Lemma \ref{pot-decay}]
Let us invoke \eqref{j}. Then since $xu \in C(\mathbb R; L_x^2)$, we have $\mathbf{J} = x+2it\nabla$. From this and Lemma \ref{loc-virial} we deduce the pseudo-conformal identity:
	\begin{align*}
	\frac{d}{dt}\int [|\mathbf J u|^2 + 8t^2V(u)]\,dx
	&= -4t \left[ -\frac12\int x\cdot \nabla K_1|u|^4\,dx + \frac13\int(2K_2-x\cdot \nabla K_2)|u|^6 \,dx\right].
\end{align*}
Since $0 \le b_2 \le 2+b_1$ we obtain
\begin{align*}
\frac{d}{dt}\int [|\mathbf J u|^2 + 8t^2V(u)]\,dx
	&\le -4t \left[-\frac{b_1}{2}\int |x|^{b_1}|u|^4\,dx + \frac{2-b_2}{3}\int|x|^{b_2}|u|^6 \,dx\right]\\
	&\le 8b_1tV(u(t)).
	\end{align*}
Integrating this over $[0, t]$, energy conservation gives us
	\begin{align*}
	t^2V(u(t)) &\le \frac18\||x|\varphi\|_{L_x^2}^2 + b_1\int_0^t \tau V(u(\tau))\,d\tau\\
	&\le \frac18\||x|\varphi\|_{L_x^2}^2 + b_1\int_0^1 \tau V(u(\tau))\,d\tau + b_1\int_1^t\tau V(u(\tau))\,d\tau\\
	&\le C(m(\varphi), E(\varphi), \||x|\varphi\|_{L_x^2})) + b_1\int_1^t\tau V(u(\tau))\,d\tau.
	\end{align*}
Then from Gronwall's type inequality it follows that
	$$
	t^2V(u(t)) \le C(m(\varphi), E(\varphi), \||x|\varphi\|_{L_x^2}))\exp\left[\int_1^t \frac{b_1}{\tau}\,d\tau\right] = C(m(\varphi), E(\varphi), \||x|\varphi\|_{L_x^2}))\;t^{b_1}.
	$$
	This completes the proof of Lemma \ref{pot-decay}.
\end{proof}

\section*{Acknowledgments}
The authors would like to thank the anonymous referees for their careful reading and valuable comments for this paper. This work was supported by NRF-2018R1D1A3B07047782 (Republic of Korea).

\end{document}